\documentclass[10pt]{article}
\usepackage[all]{xy}
\usepackage{amsfonts,amsmath,oldgerm,amssymb,amscd}
\newcommand{\ra}{\rightarrow}		
\newcommand{\by}[1]{\stackrel{#1}{\ra}}

		\newcommand{\wt}{\widetilde}
\newcommand{\iso}{\by \sim}

\newtheorem{theorem}{Theorem}[section]

\newtheorem{lemma}[theorem]{Lemma}

\newtheorem{corollary}[theorem]{Corollary}

\newcommand{\ga}{\alpha}	
		\newcommand{\gd}{\delta}
	
\newcommand{\gj}{\blacksquare}	
	
		\newcommand{\gs}{\sigma}
\newcommand{\gt}{\theta}

\newcommand{\BA}{\mbox{$\mathbb A$}}

	\newcommand{\BN}{\mbox{$\mathbb N$}}
	
\newcommand{\BQ}{\mbox{$\mathbb Q$}}

	\newcommand{\BZ}{\mbox{$\mathbb Z$}}

\newcommand{\CS}{\mbox{$\mathcal S$}}

	\newcommand{\p}{\mbox{$\mathfrak p$}}
\newcommand{\mq}{\mbox{$\mathfrak q$}}

	\newcommand{\op}{\mbox{$\oplus$}}
\newcommand{\Spec}{\mbox{\rm Spec\,}}	
	
\newcommand{\Aut}{\mbox{\rm Aut\,}}	
\newcommand{\Hom}{\mbox{\rm Hom\,}}	

\newcommand{\Um}{\mbox{\rm Um}}		\newcommand{\SL}{\mbox{\rm SL}}
\newcommand{\GL}{\mbox{\rm GL}}		
	        \newcommand{\E}{\mbox{\rm E}}

\begin{document}
\begin{center}
{\Large \bf Stability results for projective modules over Rees algebras}\\\vspace{.2in} {\large
  Ravi A. Rao and Husney Parvez Sarwar }\\
\vspace{.1in}
{\small
School of Mathematics, T.I.F.R. Mumbai, Navy Nagar, Mumbai - 400005, India;\;\\
    ravi@math.tifr.res.in, mathparvez@gmail.com}
\end{center}

\begin{abstract}
 We provide a class of commutative Noetherian domains $R$ of dimension $d$ such that
 every finitely generated projective $R$-module $P$ of rank $d$ splits off a free summand of rank one.
 On this class, we also show that $P$ is cancellative. At the end we give some applications to the 
 number of generators of a module over the Rees algebras.
 \end{abstract}
 
 

\section{Introduction}

Let $R$ be a commutative Noetherian ring of Krull dimension $d$. A classical result of Serre \cite{Serre} says that 
every finitely generated projective $R$-module $P$ of rank $>d$ splits off a free summand. This is the  best possible
result in general as it is evidenced from  the well-known example of ``the tangent bundle over real algebraic sphere of dimension two".
Therefore the question  ``splitting off a free summand" becomes subtle when $rank(P)=d$. 
If $R$ is a reduced affine algebra over an algebraically
closed field, then for a rank $d$ projective $R$-module $P$, M.P. Murthy \cite{Mu94} defined an obstruction class $c_d(P)$ in
the group $F^dK_0(A)$. Further assuming $F^dK_0(A)$ has no $(d-1)!$ torsion, he proved that $c_d(P)=0$
if and only if $P$ splits off a free summand of rank one.

For a commutative Noetherian ring $R$ of dimension $d$, Bhatwadekar--Raja Sridharan
(\cite{BR98},\cite{BR00})
defined an obstruction group called Euler class group, denoted by $\E^d(R)$. Assume $\BQ \subset R$, then
given a projective $R$-module of rank $d$, Bhatwadekar--Raja Sridharan defined an obstruction class $e_d(P)$
and  proved $e_d(P)=0$ in $\E^d(R)$ if and only if $P$ splits off a free summand of rank one.
Later, for a smooth scheme $X$ of dimension $n$, Barge--Morel \cite{BM}
defined the Chow--Witt group $\wt{CH}^j(X)$ ($j\geq 0$) and  associated to each vector bundle $E$ of rank $n$ with trivial determinant an Euler class
$\wt{c_n}(E)$ in $\wt{CH}^n(X)$. Let $A$ be a smooth affine domain of dimension $n$ and $P$ a finitely generated projective $A$-module of rank $n$.
Then it was proved that $\wt{c_n}(P)=0$ if and only if $P\iso Q\op A$ for $n=2$ in \cite{BM} (see also  \cite{F08}), $n=3$ in
\cite{FS09} and $n\geq 4$ in \cite{M012}. 

A recent result of Marco Schlichting \cite{Marco16} proved a similar kind of result for a commutative Noetherian ring $R$
of dimension $d$ all of whose residue fields are infinite. Precisely, given a rank $d$ oriented projective $R$-module $P$,
he defined a class $e(P)$ in $H^d_{Zar}(R, \mathcal{K}_d^{MW})$ such that $e(P)=0$ if and only if 
 $P$ splits off a free summand of rank one.
 
 One of the aims  of this  article is to provide a class of examples of commutative Noetherian rings $R$ of dimension $d$
 such that every rank $d$ projective $R$-module splits off a free summand of rank one. We prove the following:
 
 \begin{theorem}\label{unim}
   Let $R$ be a commutative Noetherian domain of dimension $d-1$ ($d\geq 1$) and $I$ an ideal of $R$.
   Define $A:=R[It] $ or $R[It, t^{-1}]$ (note that $dim(A)\leq d$).
  Let $P$ be a projective $A$-module  of rank  $d$. Then  $P\iso Q\op A$ for some projective $A$-module $Q$.
 \end{theorem}

 In particular, if $\BQ \subset A$, then the obstruction class $e_d(P)$ defined by Bhatwadekar--Raja Sridharan \cite{BR00}
 is zero in $\E^d(A)$. Also, in the view of Schlichting's result \cite{Marco16}, 
 if we assume that all residue fields of $A$ are infinite,
 then $e(P)$ defined by Schlichting is zero in $H^d_{Zar}(A, \mathcal{K}_d^{MW})$. For $A=R[t]$ and for birational
 overrings of $R[t]$, a similar type of result is proved by Plumstead \cite{P} and Rao (\cite{Rao82}, \cite{RaoTh}) respectively.
 
 In this direction, a parallel problem is ``the cancellation problem". Let $P$ be a projective module over a commutative
 Noetherian ring $R$ of dimension $d$ such that $rank(P)>d$. Then Bass \cite{Bass64} proved that $P$ is cancellative
 i.e. $P\op Q \iso P'\op Q \Rightarrow P\iso P'$. Again this is the best possible result in general as it is
 evidenced by the same well-known example ``tangent bundle over the real algebraic sphere of dimension two". However Suslin 
 (\cite{Su77b}) proved that if $R$ is an affine algebra of dimension $d$ over an algebraically closed field,
 then every projective $R$-module of rank $d$ is cancellative. We enlarge the class of rings by proving the 
 following result.
 
 \begin{theorem}\label{cancel}
  Let $R$ be a commutative Noetherian domain of dimension $d-1$ ($d\geq 1$) and $I$ an ideal of $R$.
   Define $A:=R[It] $ or $R[It, t^{-1}]$.
  Then every finitely generated projective $A$-module  of rank  $d$ is cancellative.
 \end{theorem}
 
  For $A=R[t]$ and for birational
 overrings of $R[t]$, a  similar type of result is proved by Plumstead \cite{P} and Rao (\cite{Rao82}, \cite{RaoTh}) respectively.
 
 The following result follows from our  result  Theorem \ref{cancel} and a result of Wiemers \cite[Theorem]{Wi1}.
 
 \begin{corollary}
  Let $R$ be a commutative Noetherian domain of dimension $d-1$ ($d\geq 1$) such that $1/d! \in R$ and $I$ an ideal of $R$.
   Define $A:=R[It] $ or $R[It, t^{-1}]$.
  Then every finitely generated projective $A[X_1,\ldots,X_n]$-module  of rank  $d$ is cancellative.
 \end{corollary}

As an application of our result, we prove the following result.

\begin{theorem}
 Let $R$ be a commutative Noetherian domain of dimension $d$ and $I$ an ideal of $R$. Let $M$ be  a finitely
 generated module over   $ A:=R[It]$ or $R[It, t^{-1}]$. 
 Then $M$ is generated by $e(M):= Sup_{\p}\{\mu_{\p}(M)+dim(A/\p)\}$
 elements.
\end{theorem}

Let $A$ be a domain of dimension $n$, $R=A[X_1,\ldots,X_m]$ and $I$ the ideal of $R$ generated by $(X_1,\ldots,X_m)$.
Then $R[It]=A[X_1,\ldots,X_m,X_1t,\ldots,X_mt]$. Note that in this case $R[It]$ becomes a monoid algebra $A[M]$,
where $M$ is the monoid generated by $(X_1,\ldots,X_m,X_1t,\ldots,X_mt)$. In this case Gubeladze \cite{G1} conjectured
that every projective $A[M]$-module of rank $>n$ splits off a free summand of rank one. In  Theorem \ref{unim}, we have verified it affirmatively  but our
rank-dimension condition is not optimal. We note that the second author and  Keshari \cite{KeS} studied the problem of existence of
unimodular elements over monoid algebras. But their results do not cover the above monoid algebra.

{\bf Acknowledgement:} We would like to thank the referee for carefully reading the paper and  some useful comments.

 \section{Notations, Rees algebras and some properties} 
   {\it Throughout the paper, we assume that all the modules are finitely generated.}
   
   Let $A$ be a commutative ring and $Q$ a  $A$-module. We say $p\in Q$ is unimodular
   if the order ideal $O_Q(p)=\{\phi(p) \,|\,\, \phi \in Q^*=\Hom(Q,A)\}$ equals
   $A$. Let $\p \in \Spec(A)$. An element $q\in Q$ is said to be a basic element of $Q$ at
   $\p$ if $q \notin \p Q_{\p}$. We say $q$ is a basic element of $Q$ if it is a basic element of $Q$
   at every prime ideal of $A$.  Let $\mu_{\p}(Q)$ denote the minimum number of generators of $Q_{\p}$
   over $A_{\p}$.
   Let $P$ be a projective $A$-module. We say $P$ is cancellative if
   $P\op Q\cong P' \op Q\Rightarrow P\cong P'$  for all  projective $A$-modules $P'$ and $Q$.
   
   The set of all unimodular elements in $Q$ is denoted by
   $\Um(Q)$. We write $\E_n(A)$ for the group generated by the set of all
   $n\times n$ elementary matrices over $A$ and $\Um_n(A)$ for
   $\Um(A^n)$. We denote by $\Aut_A(Q )$, the group of all $A$-automorphisms
   of $Q$.
   
   For an ideal $J$ of $A$, we denote by $\E(A \oplus Q, J)$, the subgroup
   of $\Aut_A (A\oplus Q )$ generated by  the automorphisms
   $\Delta_{a\phi}= \bigl( \begin{smallmatrix} 1 & a\phi \\ 0&id_Q
   \end{smallmatrix} \bigr)$ and $\Gamma_{q}= \bigl( \begin{smallmatrix}
   1 & 0 \\ q&id_Q
   \end{smallmatrix} \bigr)$ with $a\in J$, $\phi \in Q^*$ and $q\in JQ$. 
   Further, we shall write $\E(A\oplus Q)$ for $\E(A\oplus Q, A)$.
   We denote by $\Um(A\oplus Q, J)$, the set of all $(a,q) \in \Um(A\oplus
   Q )$ with $a \in 1 + J$ and $q \in JQ$.
   
   {\bf Generalized dimension:} Let $R$ be a commutative ring and $\CS \subset \Spec(R)$.
   Let $\gd: \CS\ra \BN\cup \{0\}$ be a function. Define a partial order on $\CS$
   as $\p<<\mq$ if $\p\subset \mq$ and $\gd(\p)>\gd(\mq)$. We say that $\gd$ is
   a generalized dimension function on $\CS$ if for any ideal $I$ of $R$,
   $V(I)\cap \CS$ has only a finite number of minimal elements with respect to 
   $<<$. We say that $R$ has the generalized dimension $d$ if $d=min_{\gd}(max_{\p\in \Spec(R)} \gd(\p))$. The notion
   of the generalized dimension was introduced by Plumstead in \cite{P}.
   
   For example, the standard dimension function $\gd(\p)=coheight(\p):=\dim(R/\p)$ is a generalized dimension function.
   Thus, the generalized dimension of $R$ is $\leq$ the Krull dimension of $R$.
   Observe that if $s\in R$ is such that $R/(s)$ and $R_s$ have the generalized dimension $\leq d$, then the generalized
   dimension of $R\leq d$. Indeed, if $\gd_1$ and $\gd_2$ are generalized dimension functions on $R/sR$ and $ R_s$ respectively
   with $\gd_i\leq d$. Then we define $\gd:\Spec(R)\ra \BN\cup \{0\}$ as follows $\gd(\p)=\gd_1(\p)$ if $s\in \p$ and
   $\gd(\p)=\gd_2(\p)$ if $s\notin \p$. Now clearly $\gd$ is a generalized dimension function on $R$ with $\gd \leq d$.
   
   We note down an example of 
   Plumstead \cite[Example 4]{P} where he has given a ring having generalized dimension $<$ Krull
   dimension. Take $A:=R[X]$, where $R$ is a ring having an element $s\in rad(R)$ with $\dim(R/sR)< \dim(R)$.
   Then Plumstead \cite{P} proved that the generalized dimension of $A<\dim(A)$.
   
   {\bf The Rees algebras and the extended Rees algebras:} Let $R$ be a commutative Noetherian ring of dimension $d$ and $I$
   an ideal of $R$. Then the algebra
   $$R[It]:=\{\sum_{i=0}^{n}a_it^i: n\in \BN, a_i\in I^i\}=\op_{n\geq0}I^nt^n$$ 
   is called the Rees algebra of $R$ with respect to $I$. Sometimes it is also called the blow-up algebra. These
   algebras arise naturally in the process of blowing-up a variety along a subvariety. In the case of an affine variety
   $V(I)\subset \Spec(R)$,  the blowing-up is the natural map from $Proj(R[It])\ra \Spec(R)$. It is a fact that
   dimension of $R[It]\leq d+1$ (see \cite[Theorem 1.3]{Va}). Further if $I$ is not contained in any minimal primes of $R$, then the dimension
   of $R[It]=d+1$ (see \cite[Theorem 1.3]{Va}). For further properties of $R[It]$, we refer the reader to \cite{Va}.
   
   One defines the extended Rees algebra $R[It, t^{-1}]$ with respect to an ideal $I$ of $R$ as  a subring of $R[t,t^{-1}]$
    as follows
   $$R[It, t^{-1}]=\{\sum_{i=-n}^{n}a_it^i: n\in \BN, a_i\in I^i\}=\op_{n\in \BZ}I^nt^n,$$
   where $I^n=R$ for $n\leq0$.
   It is easy to observe that $R[It, t^{-1}]$ is birational
   to $R[It]$, hence the dimension of $R[It, t^{-1}]\leq d+1$.

\section{Splitting off and cancellation results for  Rees algebras}
 
 \begin{lemma}\label{mcset}
 	Let $R$ be a commutative ring, $S\subset R$ be a multiplicative subset and $I\subset R$ an
 	ideal.  Then $S^{-1}(R[It])=S^{-1}R[(S^{-1}I)t]$.
 \end{lemma}
 \begin{proof}
 	By definition $S^{-1}(R[It])=S^{-1}R \oplus S^{-1}(IR)t\oplus S^{-1}(IR)^2t^2\oplus \cdots$ and 
 	$$S^{-1}R[(S^{-1}I)t]=S^{-1}R \oplus S^{-1}I(S^{-1}R)t\oplus (S^{-1}I(S^{-1}R))^2t^2\oplus \cdots .$$
 	 Since localization commutes with direct sums,  we have 
 	$S^{-1}(R[It])=S^{-1}R[(S^{-1}I)t]$.
 	$\hfill \gj$
 	\end{proof}
 
  \begin{lemma}\label{lem:gen-dim}
 	Let $A$ be a commutative Noetherian ring of dimension $d\geq 1$ and let
 	$s$ be a non-zero-divisor of $A$. 
 	Then the generalized dimension of $A_{1+sA}$ is $\leq d-1$.
 \end{lemma}
 \begin{proof}
 	Let $\mathcal{P}_1$ be the set of all primes of  $A_{1+sA}$ which contains $s$ and 
 	$\mathcal{P}_2:=\{\mathfrak{p} \in {\rm Spec}(A_{1+sA}): ht(\mathfrak{p)}<d\}$. We claim that
 	$${\rm Spec}(A_{1+sA})=\mathcal{P}_1\cup \mathcal{P}_2.$$ 
 	Let $\mathfrak{p}\in {\rm Spec}(A_{1+sA})$ be a prime ideal of height $d$. Hence $\mathfrak{p}$ is
 	a maximal ideal of $A_{1+sA}$. We claim that $s\in \p$. Suppose not, then $\p+(s)=(1)$.  This implies
 	that there exists $p\in \p$ such that $ap= 1+sb$. But $1+sb$ is a unit in $A_{1+sA}$ which is a 
 	contradiction to the fact $1+sb=pa\in \p$.
 	This establishes the claim.
 	Now following \cite[Example 2]{P},
 	we get that the generalized dimension of $A_{1+sA}$ is $\leq d-1$.	
 	$\hfill \gj$
 \end{proof}
 \begin{lemma}(Plumstead \cite{P})\label{lem:plum}
 	Let A be a commutative Noetherian ring of generalized dimension $d$ and $P$ a projective
 	$A$-module of rank $\geq d+1$. Then 
 	
 	$(1)$
 	$P$ has a unimodular element.  
 	More generally, if $M$ is  a finitely generated $A$-module such that $\mu_{\p}(M)\geq d$
 	for all $\p \in {\rm Spec}(A)$, then $M$ has a basic element.
 	
 	$(2)$ $P$ is cancellative. In fact $\E(A\oplus P )$ acts transitively on $\Um(A\oplus P)$. 
 \end{lemma}
 \begin{proof}
 	This is an observation made by Plumstead in \cite[page 1421, paragraph 4]{P} except part $(2)$
 	second statement. For this, let $(a,p)\in \Um(A\oplus P)$. 
 	By Eisenbud--Evans Theorem (see the version in \cite[\S 1]{P}),
 	there exists an element $q\in P$ such that $p+aq$ is a unimodular element. Hence there exists $\psi \in P^*$ such that $\psi(p+aq)=1$.
 	Now we have  $\Gamma_{-aq-p}\Delta_{\psi}\Delta_{-a\psi}\Gamma_{q}(a,p)^t=(1,0)^t$ ($v^t $=transpose of the vector $v$). This finishes the proof.
 	$\hfill \gj$
 \end{proof}
 
 \begin{lemma}\label{lem:gen-dim-res}
 	Let $R$ be a commutative Noetherian ring of dimension $d$ and let $A:=R[It]$ or $R[It,t^{-1}]$.
 	Let  $s$ be a non-zero-divisor of $R$ and $P$ a projective $A_{1+sA}$-module of rank $\geq d+1$. Then 
 	
 	$(1)$ $P$ has a unimodular element. More generally, if $M$ is  a finitely generated $A$-module such that $\mu_{\p}(M)\geq d$
 	for all $\p \in {\rm Spec}(A)$, then $M$ has a basic element.
 	
 	$(2)$ $P$ is cancellative. More generally, $\E(A\oplus P )$ acts transitively on $\Um(A\oplus P)$.

 \end{lemma}
 \begin{proof}
 	This  is a direct consequence of  Lemma \ref{lem:gen-dim} and  Lemma \ref{lem:plum}.
 	$\hfill \gj$
 \end{proof}

\begin{theorem}\label{thm:unimodular}
 Let $R$ be a commutative Noetherian domain of dimension $d$ and $I$ an ideal of $R$.  
 
 $(1)$ Let $P$ be a projective module over the Rees algebra $R[It]$ such that the rank of $P$ is $>d$. Then  $P$
 has a unimodular element.
 
 $(2)$ Let $Q$ be a projective module over the extended Rees algebra $R[It, t^{-1}]$ such that the  rank
  of  $Q$ is $>d$. Then $Q$ has a unimodular element.
\end{theorem}
\begin{proof}
$(1)$
Let $A=R[It]$. Since $R$ is a domain, so is $R[t]$. Since $A$ is a subring of $R[t]$, we conclude
that $A$ is a domain.
If $I=(0)$, then $A=R$. In this case, the result follows from a classical result of Serre \cite{Serre}. 
If $I=(1)$, then $A=R[t]$. In this case, the result follows from \cite[Corollary 4]{P}.
So, we assume that $(0)\neq I\neq (1)$. We already noted that $dim(A)\leq d+1$.
In fact here $dim(A)=d+1$ since $I\neq 0$. Hence in the view of a classical
result of Serre \cite{Serre}, we only have to consider the case when the rank of $P$ is $d+1$.
 
 Let $S$ be the set of all non-zero-divisors of $R$. Then by Lemma \ref{mcset}, we have 
  $S^{-1}R[It]\cong R'[t]$, where $R'$ is
 a field. Therefore $S^{-1}P$ is a free $S^{-1}A$-module. Since $P$ is finitely generated, there exists $s\in S$ such that $P_s$ is a free $A_s$-module.
 
Let $a$ be  a non-zero non-unit element of $I$. Let us consider the following commutative diagram of rings

\[
   	\xymatrix{
   		R[It] \ar@{->}[r]^{p_1}
   		\ar@{->}[d]^{p_2}
   		&R[It]_{as}\cong R_{as}[t]
   		\ar@{->}[d]^{j_1}
   		\\
   		R[It]_{1+asR[It]} \ar@{->}[r]^{j_2}
   		&(R_{as}[t])_{1+asR[It]}.
   	}
   	\]
   	Since $a\in I$, we have $IR_a=R_a$. Hence we get  the isomorphism $R[It]_{as}\cong R_{as}[t]$ 
   	using Lemma \ref{mcset}. 
   	It is easy to see that ${\rm Spec}(R[It])={\rm Spec}(R[It]_{as}) \cup {\rm Spec}(R[It]_{1+asR[It]})$.
   	Hence the above commutative diagram of rings is a localization Cartesian square of rings.
Since $P_{as}\cong R[It]_{as}^{d+1}$, $P_{as}$ has a unimodular element $p_1$. 
Since $\dim(R[It])\leq d+1$, $\dim(R[It]/asR[It])\leq d$ (recall that $as$ is  a non-zero-divisor in $R[It]$  since $R[It]$ is a domain). By Lemma \ref{lem:gen-dim-res},
$P_{1+asR[It]}$ has a unimodular element $p_2$. Hence we can write
$P_{1+asR[It]}\cong R[It]_{1+asR[It]}\op Q$,
where $Q$  is a  projective module over $R[It]_{1+asR[It]}$.

Observe that $B:=(R_{as}[t])_{1+asR[It]}=T^{-1}R'[t]$,
where $R'=R_{as(1+asR)}$ is of dimension $d-1$ which follows from the arguments of Lemma \ref{lem:gen-dim} and $T\subset R$ is a multiplicatively closed set. Hence  we have $P_{as(1+asR[It])}\cong B^{d+1}$. Let the images of $p_1$ and $p_2$ be the unimodular
 vectors $\overline u$ 
and $\overline v$ respectively.
By a result of Rao \cite[Theorem 5.1(I)]{Rao82} (see also \cite[Theorem 6.2]{BRoy}), there exists $\gs\in \E_{d+1}(B)$ such that
 $\gs(\overline{u})=(\overline{v})$. Now by \cite[Corollary 3.2]{Roy82} (see for more explanations \cite[Proposition 3.2]{BLR}), we can write
$\gs=(\gs_1)_{as}(\gs_2)_{1+asR[It]}$, where $\gs_1\in \E_{d+1}(R[It]_{as})$, $\gs_2\in \E(R[It]_{1+asR[It]}\op Q)$.
Hence suitably changing $p_1$, $p_2$, we can assume that $\overline{u}=\overline{v}$.
 Therefore by a sheaf patching on a two cover, we get
a unimodular element $p$ in $P$. Now we explain this  more algebraically. The above Cartesian
square of rings induces the following Cartesian square of projective modules

\[
\xymatrix{
	P \ar@{->}[r]^{p_1}
	\ar@{->}[d]^{p_2}
	&P_{as}
	\ar@{->}[d]^{j_1}
	\\
	P_{1+asR[It]} \ar@{->}[r]^{j_2}
	&P_{as(1+asR[It])}.
}
\]
Having a unimodular element in $P$ is equivalent to an existence of a surjective 
homomorphism $\phi: P \rightarrow R[It]$. The above arguments show that we have a surjection locally, hence by the Cartesian square diagram, we get
a surjection globally.
 This finishes the proof of $(1)$.

$(2)$ Proof of this part is  a verbatim copy of the proof of Part $(1)$.
$\hfill \gj$
\end{proof}


We prove the following criterion for cancellation.

\begin{lemma}\label{lem:criteria}
 {\bf (Criterion for cancellation)} Let $R$ be a commutative Noetherian ring of dimension $d$ and $P$ a projective
 $R$-module of rank $n$. Let $s,t\in R$ such that $sR+tR=R$. 
 Let $(b,p)\in \Um(R\op P)$.
 Assume that there exist
 $\gs_1\in \Aut(R_s\op P_s)$, $\gs_2\in \Aut(R_t\op P_t)$ such that
 $((b,p)_{s})\gs_1=(1,0)$,  $((b,p)_t)\gs_2=(1,0)$ respectively.
 
 $(1)$  If we  define $\eta=(\gs_1)_t^{-1}(\gs_2)_{s}$, then
 $\eta= \bigl( \begin{smallmatrix} 1 & 0 \\ * & \gt
\end{smallmatrix} \bigr)$, where $\gt \in \Aut{(P_{st})}$.

$(2)$ Further if 
  $\eta=(\eta_2)_t(\eta_1)_s$, where $\eta_2 = \bigl( \begin{smallmatrix} 1 & 0 \\ *_2 & \gt_s
 \end{smallmatrix} \bigr) \in \Aut(R_{s}\op P_{s})$ and  $\eta_1=  \bigl( \begin{smallmatrix} 1 & 0 \\ *_1 & \gt_t
 \end{smallmatrix} \bigr) \in \Aut(R_{t}\op P_{t})$, then there exists $\phi \in \Aut(R\op P)$ such that $(b,p)\phi=(1,0)$.
\end{lemma}

\begin{proof}
	For $(1)$, we observe that
 $(1,0)\eta=(1,0)$. Hence it is easy to see that
  $\eta= \bigl( \begin{smallmatrix} 1 & 0 \\ * & \gt
\end{smallmatrix} \bigr)$, where $\gt \in \Aut(P_{st})$.
 Consider the following Cartesian square
 \[
   	\xymatrix{
   		R \ar@{->}[r]^{p_1}
   		\ar@{->}[d]^{p_2}
   		&R_s
   		\ar@{->}[d]^{j_1}
   		\\
   		R_t \ar@{->}[r]^{j_2}
   		& R_{st}.     
   	}
   	\]
  
 We set  $\gs_s=\gs_1\eta_2$ and $\gs_t=\gs_2\eta_1^{-1}$. Then we observe that $((b,p)_{s})\gs_s=(1,0)$,  $((b,p)_t)\gs_t=(1,0)$
and $(\gs_s)_t=(\gs_t)_s$ (recall that $\eta=(\gs_1)_t^{-1}(\gs_2)_{s}=(\eta_2)_t(\eta_1)_s$).
Now by
a standard  patching argument, we will get an automorphism $\phi\in \Aut(R\op P)$ such that $(b,p)\phi=(1,0)$. Now we explain how we get the automorphisn $\phi$. The following is a commutative
diagram of modules where the front square and the back square are  Cartesian.

 \[
\xymatrix{
	R\oplus P \ar@{->}[rr]\ar@{->}[rd]^{\phi}\ar@{->}[dd] &&R_s\oplus P_s \ar@{->}[dd]\ar@{->}[rd]^{\phi_s}
	\\
	&R\oplus P \ar@{->}[rr]	\ar@{->}[dd] &&R_s\oplus P_s \ar@{->}[dd]
	\\
	R_t\oplus P_t \ar@{->}[rr]\ar@{->}[rd]^{\phi_t}& & R_{st}\oplus P_{st}\ar@{->}[rd]^{\phi_{st}}     
	\\
 &R_t\oplus P_t \ar@{->}[rr] && R_{st}\oplus P_{st}     
}
\]
We get the homomorphism $\phi$ by the universal property of the Cartesian square. We observe that $\phi$
is  locally an isomorphism and it sends $(b,p)$ locally to $(1,0)$. Hence we conclude that $\phi$ is an isomorphism
and it sends $(b,p)$ to $(1,0)$.
$\hfill \gj$
\end{proof}

\begin{lemma}\label{lem:splitting}
	Continuing with the notations as in Lemma \ref{lem:criteria},  we further assume that $P_s$ is free.
	Then we can write  $\eta=(\eta_2)_t(\eta_1)_s$, where $\eta_2 = \bigl( \begin{smallmatrix} 1 & 0 \\ *_2 & \gt_s
	\end{smallmatrix} \bigr) \in \Aut(R_{s}\op P_{s})$ and  $\eta_1=  \bigl( \begin{smallmatrix} 1 & 0 \\ *_1 & \gt_t
	\end{smallmatrix} \bigr) \in \Aut(R_{t}\op P_{t})$ i.e. $(1,0)\eta_i=(1,0)$ for $i=1,2$.
	
	\end{lemma}
\begin{proof}
	By \cite[Corollary 3.2]{Roy82}, we can write
	$\eta=(\eta_1)_t(\eta_2)_{s}$ such that $\eta_1\in \Aut(R_s\op P_{s})$ and $\eta_2\in \Aut(R_t\op P_t)$.  Since $(1,0)\eta=(1,0)$,  the explicit computations of $\eta_1$, $\eta_2$ in the proof of \cite[Proposition 3.1]{Roy82} can be modified suitably to get $\eta=(\eta_1)_t(\eta_2)_{s}$
  such that	$(1,0)\eta_1=(1,0)$ and $(1,0)\eta_2=(1,0)$.
	Hence we have  $\eta_1=\bigl( \begin{smallmatrix} 1 & 0 \\ *_1 & \gt_1
	\end{smallmatrix} \bigr) \in \Aut(R_{s}\op P_{s})$ and $\eta_2=\bigl( \begin{smallmatrix} 1 &  0\\ *_2 & \gt_2
	\end{smallmatrix} \bigr) \in\Aut(R_t\op P_t) $.
	Alternatively, one can use Quillen splitting lemma to get the required splitting as follows.
	Consider the elementary group $\E(R_{st}[Z]\oplus P_{st}[Z])$, where $Z$ is  a variable.
	It is easy to see that we can find $\alpha(Z) \in \E(R_{st}[Z]\oplus P_{st}[Z])$ with the properties
	$(1,0) \alpha(Z) = (1,0)$,  $\alpha(1) = \eta$, and $\alpha(0) = Id$. By 
	Quillen's splitting lemma \cite[Theorem 1, paragraph 2]{Qu}, for $g = (s)^N$ with large $N$, we have 
	\begin{eqnarray*}
		\alpha(Z) &=& (\alpha(Z)\alpha(gZ)^{-1})_t\alpha(gZ)_{s},
	\end{eqnarray*}
	with $\alpha(Z)\alpha(gZ)^{-1} \in Aut(R_{s}[Z]\oplus P_{s}[Z])$, $\alpha(gZ) \in \Aut(R_t[Z]\oplus P_t[Z])$. 
	Note that $(1,0)\alpha(gZ) = (1,0)$; hence $(1,0)\{(\alpha(Z)\alpha(gZ)^{-1})\}= (1,0)$. 
	Specializing $Z = 1$ gives the required splitting of $\eta = (\eta_1)_t(\eta_2)_{s}$, 
	with $\eta_1 \in \Aut(R_t\oplus P_t)$, $\eta_2 \in \Aut(R_{s}\oplus P_{s})$ such that
	$(1,0)\eta_i  = (1,0)$  for $i = 1, 2$. 
	$\hfill \gj$
	\end{proof}

\begin{theorem}\label{thm:canc}
 Let $R$ be a commutative Noetherian domain of dimension $d$ and $I$ an ideal of $R$. 
 
 $(1)$ Let $P$ be a projective module over the Rees algebra $R[It]$ such that $rank(P)>d$. 
 Let $(b,p)\in \Um(R[It]\op P)$. Then  there exists $\gs \in \Aut(R[It]\op P)$ such that $(b,p)\gs=(1,0)$.

 $(2)$  Let $P$ be a projective module over the extended  Rees algebra $R[It,t^{-1}]$ such that $rank(P)>d$. 
 Let $(b,p)\in \Um(R[It,t^{-1}]\oplus P)$. Then  there exists $\gs \in \Aut(R[It, t^{-1}]\op P)$ such that $(b,p)\gs=(1,0)$.
 \end{theorem}

\begin{proof}
	
	Let $A=R[It]$. Since $R$ is a domain, so is $R[t]$. Since $A$ is a subring of $R[t]$, we conclude
	that $A$ is a domain.
	If $I=(0)$, then $A=R$. In this case, the result follows from a  result of Bass \cite{Bass64}. 
	If $I=(1)$, then $A=R[t]$. In this case, the result follows from \cite[Corollary 2]{P}. So, we assume that $(0)\neq I\neq (1)$.
	 Note that in this case $\dim(A)= d+1$. Hence in the view of a classical
	result of Bass \cite{Bass64}, we only have to consider the case when the rank of $P$ is $d+1$.

	Let $S$ be the set of all non-zero-divisors of $R$. Then using Lemma \ref{mcset},  we have $S^{-1}R[It]\cong R'[t]$ where $R'$ is
	a field. Therefore $S^{-1}P$ is a free $S^{-1}A$-module. Since $P$ is finitely generated, there exists $s\in S$ such that $P_s$ is a free $A_s$-module.

 Let $(b,p)\in \Um(R[It]\op P)$ and $a$ a non-zero non-unit element of $I$. We have already observed the following Cartesian square of rings in the proof of  Theorem \ref{thm:unimodular}
 \[
   	\xymatrix{
   		R[It] \ar@{->}[r]^{p_1}
   		\ar@{->}[d]^{p_2}
   		&R[It]_{as}\cong R_{as}[t]
   		\ar@{->}[d]^{j_1}
   		\\
   		R[It]_{1+asR[It]} \ar@{->}[r]^{j_2}
   		&(R_{as}[t])_{1+asR[It]}.      
   	}
   	\]
  Let $T:=1+asR[It]$.  Since $P_{as}$ is free,  using a result of Suslin \cite[Theorem 2.6]{Su77},  there exists $\gs_1\in \E(R[It]_{as}\op P_{as})$ such that
  $((b,p)_{as})\gs_1=(1,0)$ for $d\geq 1$. For $d=0$, $R$ becomes a field. Also in this case it is easy to see that  
   there exists $\gs_1\in \E(R[It]_{as}\op P_{as})$ such that
  $((b,p)_{as})\gs_1=(1,0)$. 
  By   Lemma \ref{lem:gen-dim-res}, there exists
  $\gs_2\in \E(R[It]_T\op P_T)$ such that  $((b,p)_T)\gs_2=(1,0)$.  Let $\eta=(\gs_1)_T^{-1}(\gs_2)_{as}$.
  Since $(1,0)\eta=(1,0)$,
  $\eta= \bigl( \begin{smallmatrix} 1 & 0 \\ * & \gt
\end{smallmatrix} \bigr)$, where $\gt \in \GL_{d+1}(B)$, where $B=(R_{as}[t])_{1+asR[It]}$.
In fact $\gt \in \SL_{d+1}(B)$, since $\eta$ is elementary. 

By Lemma \ref{lem:splitting}, we have $\eta = (\eta_1)_T(\eta_2)_{as}$, 
with $\eta_1 \in \Aut(R[It]_T\oplus P_T)$, $\eta_2 \in \Aut(R_{as}[t]\oplus P_{as})$ such that
$(1,0)\eta_i  = (1,0)$  for $i = 1, 2$. 
Now by Lemma \ref{lem:criteria}, we will get an automorphism $\phi\in \Aut(R[It]\op P)$ such that $(b,p)\phi=(1,0)$.
This completes the proof of $(1)$.

$(2)$ The proof of the second part is a verbatim copy of the first part. But one needs to use the cancellation result
for the Laurent extensions (see \cite[Theorem 7.2]{Su77}).
$\hfill \gj $
\end{proof}

\section{Applications}

The following is the Eisenbud--Evans estimate on the number of generators of a module over the (extended) Rees algebras.

\begin{theorem}
 Let $R$ be a commutative Noetherian domain of dimension $d$ and $I$ an ideal of $R$. Let $M$ be  a non-zero finitely
 generated module over  $A:=R[It]$ or $R[It,t^{-1}]$. Then $M$ is generated by $e(M):= Sup_{\p}\{\mu_{\p}(M)+\dim(A/\p)\}$
 elements.
\end{theorem}
\begin{proof}
 Suppose $M$ is generated by $n$ elements such that $n>e(M)>d$. Then it is enough to prove that $M$ 
 is generated by $(n-1)$ elements. Since $M$ is generated by $n$ elements, we have the following
 surjective map $A^n\ra M\ra 0$. Let $K$ be the kernel of this map. Since $A$ is Noetherian, $K$ is finitely generated. Also, $K$ is a torsion-free module since $K$ is a submodule of $A^n$.
  Let $S$ be the  set of all non-zero divisors of $R$. Then $S^{-1}A
  \cong k[t]$, where $k$ is a field. 
Since $S^{-1}K$ is a torsion-free module over the PID $k[t]$, by the structure theorem of finite generated modules over a PID, we get
that $S^{-1}K$ is $S^{-1}A$-free. Since $K$ is finitely generated, there exists $s\in S$ such
that $K_s$ is $A_s$-free. 

Let $x_1$ be a basic element of $K_s$. We set $T=1+sA$. By  Lemma \ref{lem:plum}, we get
that $K_T$ contains a basic element $x_2$  which is a unimodular element in $A^n_T$. Now patching $x_1$ and $x_2$, we get a basic element $x$ in $K$ which is a unimodular element $x\in A^n$.
 Hence we can write $A^n\cong xA\op P$. Now we observe that locally $P$ 
surjects onto $M$. Hence $P$ surjects onto $M$ globally. Note that rank of $P\geq n-1>d$. Hence by the cancellation result Theorem \ref{thm:canc}, we have 
$P\cong A^{n-1}$. Therefore $M$ is generated by $n-1$ elements. This finishes the proof.  
$\hfill \gj$

\end{proof}

The following is a ${\rm K}_1$-analog of the above results.

\begin{theorem}
 Let $R$ be a commutative domain of dimension $d$ and $I\subset R$ an ideal of $R$. Then for $n\geq max\{3,d+2\}$, the
 natural map $\phi: \GL_n(R[It])/\E_n(R[It])\ra {\rm K}_1(R[It])$ is an isomorphism.
\end{theorem}
\begin{proof}
 Surjectivity follows from Theorem \ref{thm:canc}. We have to prove injectivity i.e. stably elementary
 is elementary. We grade $A:=R[It]$ as $A=R\op It\op I^2t^2\op \cdots$, where $A_0=R$ and $A_+=It\op I^2t^2\op \cdots$.
 Let $\ga\in \GL_n(A)$ which is stably elementary. Multiplying by elementary matrices, we can assume that 
 $\ga\in \GL_n(A,A_+)$. Let $\p \in \Spec(R)$. Note that the generalized dimension of $R_{\p}[I_{\p}t]\leq d$.
  By usual estimation, we get $\ga_{\p}\in \E_n(R_{\p}[I_{\p}t])$. Then by following Gubeladze \cite[Proposition 7.3]{G2},
we get  $\ga\in \E_n(R[It])$. This completes the proof.
 $\hfill \gj$
\end{proof}

{\small
{}

}

\end{document}